\documentclass[a4paper,10pt]{amsart}

\usepackage{amsmath,amssymb,amsfonts,latexsym}

\newtheorem{theorem}{Theorem}[section]
\newtheorem{lemma}[theorem]{Lemma}
\newtheorem{corollary}[theorem]{Corollary}

\begin{document}

\title[Some Properties of Horadam quaternions]{Some Properties of Horadam quaternions}

\author[G. Cerda-Morales]{Gamaliel Cerda-Morales}
\address{Instituto de Matem\'aticas, Pontificia Universidad Cat\'olica de Valpara\'iso, Blanco Viel 596, Valpara\'iso, Chile.}
\email{gamaliel.cerda.m@mail.pucv.cl}


\begin{abstract}
In this paper, we consider the generalized Fibonacci quaternion which is the Horadam quaternion sequence. Then we used the Binet's formula to show some properties of the Horadam quaternion. We get some generalized identities of the Horadam number and generalized Fibonacci quuaternion.

\vspace{2mm}


\vspace{2mm}


\end{abstract}


\maketitle


\section{Introduction}

Fibonacci quaternion and Lucas quaternion cover a wide range of interest in modern mathematics as they appear in the comprehensive works of \cite{Hal1,Hor1,Hor2}. The Fibonacci quaternion $Q_{F,n}$ is the term of the sequence where each term is the sum of the two previous terms beginning with the initial values $Q_{F,0}=i+j+2k$ and $Q_{F,1}=1+i+2j+3k$. The well-known Fibonacci quaternion $Q_{F,n}$ is defined as
\begin{equation}\label{Hor}
Q_{F,n}=F_{n}+iF_{n+1}+jF_{n+2}+kF_{n+3}
\end{equation} 
and the Lucas quaternion is defined as $Q_{L,n}=L_{n}+iL_{n+1}+jL_{n+2}+kL_{n+3}$ for $n\geq0$, where $F_{n}$ and $L_{n}$ are $n$-th Fibonacci and Lucas number, respectively.

Ipek \cite{Ipe} studied the $(p,q)$-Fibonacci quaternions $Q_{\mathcal{F},n}$ which is defined as 
\begin{equation}\label{Ipe}
Q_{\mathcal{F},n}=pQ_{\mathcal{F},n-1}+qQ_{\mathcal{F},n-2},\ n\geq2
\end{equation}
with initial conditions $Q_{\mathcal{F},0}=i+pj+(p^{2}+q)k$, $Q_{\mathcal{F},1}=1+pi+(p^{2}+q)j+(p^{3}+2pq)k$ and $p^{2}+4q>0$. If $p=q=1$, we get the classical Fibonacci quaternion $Q_{F,n}$. If $p=2q=2$, we get the Pell quaternion $Q_{P,n}=P_{n}+iP_{n+1}+jP_{n+2}+kP_{n+3}$, where $P_{n}$ is the $n$-th Pell number.

The well-known Binet's formulas for $(p,q)$-Fibonacci quaternion and $(p,q)$-Lucas quaternion, see \cite{Ipe}, are given by
\begin{equation}\label{Binet}
Q_{\mathcal{F},n}=\frac{\underline{\alpha}\alpha^{n}-\underline{\beta}\beta^{n}}{\alpha-\beta}\ \textrm{and}\ Q_{\mathcal{L},n}=\underline{\alpha}\alpha^{n}+\underline{\beta}\beta^{n},
\end{equation}
where $\alpha,\beta$ are roots of the characteristic equation $t^{2}-pt-q=0$, and $\underline{\alpha}=1+\alpha i+\alpha^{2} j+\alpha^{3} k$ and $\underline{\beta}=1+\beta i+\beta^{2} j+\beta^{3} k$. We note that $\alpha+\beta=p$, $\alpha \beta=-q$ and $\alpha-\beta=\sqrt{p^{2}+4q}$.

The generalized of Fibonacci quaternion $Q_{w,n}$ is defined by Halici and Karata\c{s} in \cite{Hal3} as $Q_{w,0}=a+bi+(pb+qa)j+((p^{2}+q)b+pqa)k$, $Q_{w,1}=b+(pb+qa)i+((p^{2}+q)b+pqa)j+((p^{3}+2pq)b+q(p^{2}+q)a)$ and $Q_{w,n}=pQ_{w,n-1}+qQ_{w,n-2}$, for $n\geq2$ which we call the generalized Fibonacci quaternions. So, each term of the generalized Fibonacci sequence $\{Q_{w,n}\}_{n\geq0}$ is called generalized Fibonacci quaternion.

The  Binet formula for generalized Fibonacci quaternion $Q_{w,n}$, see \cite{Hal3}, is given by
\begin{equation}\label{Binet2}
Q_{w,n}=\frac{A\underline{\alpha}\alpha^{n}-B\underline{\beta}\beta^{n}}{\alpha-\beta},
\end{equation}
where $A=b-a\beta$, $B=b-a\alpha$, $\alpha,\beta$ are roots of the characteristic equation $t^{2}-pt-q=0$, and $\underline{\alpha}=1+\alpha i+\alpha^{2} j+\alpha^{3} k$ and $\underline{\beta}=1+\beta i+\beta^{2} j+\beta^{3} k$. If $a=0$ and $b=1$, we get the classical $(p,q)$-Fibonacci quaternion $Q_{\mathcal{F},n}$. If $a=2$ and $b=p$, we get the $(p,q)$-Lucas quaternion $Q_{\mathcal{L},n}$.

In this paper, we study some properties of the $(p,q)$-Fibonacci quaternions, $(p,q)$-Lucas quaternions and the generalized Fibonacci quaternions.

\section{Main results}
There are three well-known identities for generalized Fibonacci numbers, namely, Catalan's, Cassini's, and d'Ocagne's identities. The proofs of these identities are based on Binet formulas. We can obtain these types of identities for generalized Fibonacci quaternions using the Binet formula for $Q_{w,n}$. Then, we require $\underline{\alpha} \underline{\beta}$ and $\underline{\beta}\underline{\alpha}$. These products are given in the next lemma.
\begin{lemma}
We have
\begin{equation}\label{eq:1}
\underline{\alpha}\underline{\beta}=Q_{\mathcal{L},0}-[q]-q\Delta\omega,
\end{equation}
and
\begin{equation}\label{eq:2}
\underline{\beta}\underline{\alpha}=Q_{\mathcal{L},0}-[q]+q\Delta\omega,
\end{equation}
where $\omega=qi+pj-k$, $[q]=1-q+q^{2}-q^{3}$ and $\Delta=\alpha-\beta$.
\end{lemma}
\begin{proof}
From the definitions of $\underline{\alpha}$ and $\underline{\beta}$, and using $i^{2}=j^{2}=k^{2}=-1$ and $ijk=-1$, we have
\begin{align*}
\underline{\alpha}\underline{\beta}&=2+(\alpha+\beta)i+(\alpha^{2}+\beta^{2})j+(\alpha^{3}+\beta^{3})k\\
&\ \ -(1+\alpha\beta+(\alpha\beta)^{2}+(\alpha\beta)^{3})+\alpha^{2}\beta^{2}(\beta-\alpha)i+\alpha\beta(\alpha^{2}-\beta^{2})j+\alpha\beta(\beta-\alpha)k\\
&=2+pi+(p^{2}+2q)j+(p^{3}+3pq)k-(1-q+q^{2}-q^{3})-q\Delta(qi+pj-k)\\
&=Q_{\mathcal{L},0}-[q]-q\Delta\omega,
\end{align*}
where $[q]=1-q+q^{2}-q^{3}$ and $\omega=qi+pj-k$, and the final equation gives Eq. (\ref{eq:1}). The other identity can be computed similarly.
\end{proof}

This lemma gives us the following useful identity:
\begin{equation}\label{eq:3}
\underline{\alpha}\underline{\beta}+\underline{\beta}\underline{\alpha}=2(Q_{\mathcal{L},0}-[q]).
\end{equation}

Catalan's identities for generalized Fibonacci quaternions are given in the next theorem.
\begin{theorem}
For any integers $m$ and $n$, we have
\begin{equation}\label{eq:4}
Q_{w,m}^{2}-Q_{w,m+n}Q_{w,m-n}=-AB(-q)^{m}\mathcal{F}_{-n}\left((Q_{\mathcal{L},0}-[q])\mathcal{F}_{n}-q\omega\mathcal{L}_{n}\right),
\end{equation}
where $A=b-a\beta$, $B=b-a\alpha$, and $\mathcal{F}_{n}$, $\mathcal{L}_{n}$ are the $n$-th $(p,q)$-Fibonacci and $(p,q)$-Lucas numbers, respectively.
\end{theorem}
\begin{proof}
From the Binet formula for generalized Fibonacci quaternions $Q_{w,n}$ in (\ref{Binet2}) and $\Delta^{2}=p^{2}+4q$, we have
\begin{align*}
\Delta^{2}&\left(Q_{w,m}^{2}-Q_{w,m+n}Q_{w,m-n}\right)\\
&=\left(A\underline{\alpha}\alpha^{m}-B\underline{\beta}\beta^{m}\right)^{2}- \left(A\underline{\alpha}\alpha^{m+n}-B\underline{\beta}\beta^{m+n}\right)\left(A\underline{\alpha}\alpha^{m-n}-B\underline{\beta}\beta^{m-n}\right)\\
&=AB(-q)^{m-n}\big(\underline{\alpha}\underline{\beta}\alpha^{2n}+\underline{\beta}\underline{\alpha}\beta^{2n}-(-q)^{n}\left(\underline{\alpha}\underline{\beta}+\underline{\beta}\underline{\alpha}\right)\big).
\end{align*}
We require Eqs. (\ref{eq:1}) and (\ref{eq:2}). Using this equations, we obtain
\begin{align*}
Q_{w,m}^{2}&-Q_{w,m+n}Q_{w,m-n}\\
&=\frac{AB(-q)^{m-n}}{\Delta^{2}}\left((Q_{\mathcal{L},0}-[q])(\alpha^{2n}+\beta^{2n}-2(-q)^{n})-q\Delta \omega(\alpha^{2n}-\beta^{2n}))\right)\\
&=\frac{AB(-q)^{m-n}}{\Delta^{2}}\left((Q_{\mathcal{L},0}-[q])(\mathcal{L}_{2n}-2(-q)^{n})-q\Delta^{2} \omega\mathcal{F}_{2n}\right).
\end{align*}
Using the identity $\Delta^{2}\mathcal{F}_{n}^{2}=\mathcal{L}_{2n}-2(-q)^{n}$ gives
$$Q_{w,m}^{2}-Q_{w,m+n}Q_{w,m-n}=AB(-q)^{m-n}\left((Q_{\mathcal{L},0}-[q])\mathcal{F}_{n}^{2}-q\omega\mathcal{F}_{2n}\right),$$
where $\mathcal{L}_{n}$, $\mathcal{F}_{n}$ are the $n$-th $(p,q)$-Lucas and $(p,q)$-Fibonacci numbers, respectively. With the help of the identities $\mathcal{F}_{2n}=\mathcal{F}_{n}\mathcal{L}_{n}$ and $\mathcal{F}_{-n}=-(-q)^{n}\mathcal{F}_{n}$, we have Eq. (\ref{eq:4}). The proof is completed.
\end{proof}

Taking $n=1$ in this theorem and using the identity $\mathcal{F}_{-1}=\frac{1}{q}$, we obtain Cassini's identities for generalized Fibonacci quaternions.
\begin{corollary}
For any integer $m$, we have
\begin{equation}\label{eq:5}
Q_{w,m}^{2}-Q_{w,m+1}Q_{w,m-1}=AB(-q)^{m-1}\left(Q_{\mathcal{L},0}-[q]-pq\omega\right),
\end{equation}
where $A=b-a\beta$, $B=b-a\alpha$ and $[q]=1-q+q^{2}-q^{3}$.
\end{corollary}

The following theorem gives d'Ocagne's identities for generalized Fibonacci quaternions.
\begin{theorem}
For any integers $n$ and $m$, we have
\begin{equation}\label{eq:6}
Q_{w,n}Q_{w,m+1}-Q_{w,n+1}Q_{w,m}=(-q)^{m} AB \left((Q_{\mathcal{L},0}-[q])\mathcal{F}_{n-m}-q\omega \mathcal{L}_{n-m}\right),
\end{equation}
$\mathcal{F}_{n}$, $\mathcal{L}_{n}$ are the $n$-th $(p,q)$-Fibonacci and $(p,q)$-Lucas numbers, respectively.
\end{theorem}
\begin{proof}
Using the Binet formula for the generalized Fibonacci quaternions gives
\begin{align*}
\Delta^{2}(Q_{w,n}Q_{w,m+1}&-Q_{w,n+1}Q_{w,m})\\
&=\left(A\underline{\alpha}\alpha^{n}-B\underline{\beta}\beta^{n}\right)\left(A\underline{\alpha}\alpha^{m+1}-B\underline{\beta}\beta^{m+1}\right)\\
&\ \ - \left(A\underline{\alpha}\alpha^{n+1}-B\underline{\beta}\beta^{n+1}\right)\left(A\underline{\alpha}\alpha^{m}-B\underline{\beta}\beta^{m}\right)\\
&=\Delta (-q)^{m} AB \left(\underline{\alpha}\underline{\beta}\alpha^{n-m}- \underline{\beta}\underline{\alpha}\beta^{n-m}\right).
\end{align*}
We require the Eqs. (\ref{eq:1}) and (\ref{eq:2}). Substituting these into the previous equation, we have
\begin{align*}
Q_{w,n}Q_{w,m+1}&-Q_{w,n+1}Q_{w,m}\\
&=\frac{1}{\Delta}(-q)^{m} AB \left((Q_{\mathcal{L},0}-[q])(\alpha^{n-m}-\beta^{n-m})-q\Delta\omega (\alpha^{n-m}+\beta^{n-m}) \right)\\
&=(-q)^{m} AB \left((Q_{\mathcal{L},0}-[q])\mathcal{F}_{n-m}-q\omega \mathcal{L}_{n-m} \right).
\end{align*}
The second identity in the above equality, can be proved using $\mathcal{L}_{n-m}=\alpha^{n-m}+\beta^{n-m}$ and $\Delta\mathcal{F}_{n-m}=\alpha^{n-m}-\beta^{n-m}$. This proof is completed.
\end{proof}

In particular, if $m=n-1$ in this theorem and using the identity $\mathcal{L}_{1}=p$, we obtain Cassini's identities for generalized Fibonacci quaternions. Now, taking $m=n$ in this theorem and using the identities $\mathcal{F}_{0}=0$ and $\mathcal{L}_{0}=2$, we obtain the next identity.
\begin{corollary}
For any integer $n$, we have
\begin{equation}\label{eq:7}
Q_{w,n}Q_{w,n+1}-Q_{w,n+1}Q_{w,n}=2(-q)^{n+1} AB \omega,
\end{equation}
where $A=b-a\beta$, $B=b-a\alpha$ and $\omega=qi+pj-k$.
\end{corollary}

\begin{theorem}
For any integers $n$, $r$ and $s$, we have
\begin{equation}\label{equa:7}
Q_{\mathcal{L},n+r}Q_{\mathcal{F},n+s}-Q_{\mathcal{L},n+s}Q_{\mathcal{F},n+r}=2(-q)^{n+r}\mathcal{F}_{s-r}(Q_{\mathcal{L},0}-[q]).
\end{equation}
\end{theorem}
\begin{proof}
The Binet formulas for the $(p,q)$-Lucas and $(p,q)$-Fibonacci quaternions give
\begin{align*}
\Delta(Q_{\mathcal{L},n+r}Q_{\mathcal{F},n+s}&-Q_{\mathcal{L},n+s}Q_{\mathcal{F},n+r})\\
&=\left(\underline{\alpha}\alpha^{n+r}+\underline{\beta}\beta^{n+r}\right)\left(\underline{\alpha}\alpha^{n+s}-\underline{\beta}\beta^{n+s}\right)\\
&\ \ -\left(\underline{\alpha}\alpha^{n+s}+\underline{\beta}\beta^{n+s}\right)\left(\underline{\alpha}\alpha^{n+r}-\underline{\beta}\beta^{n+r}\right)\\
&=(\alpha \beta)^{n}(\alpha^{s}\beta^{r}-\alpha^{r}\beta^{s})(\underline{\alpha}\underline{\beta}+\underline{\beta}\underline{\alpha}).
\end{align*}
Using Eqs. (\ref{eq:1}) and (\ref{eq:2}), we have
$$Q_{\mathcal{L},n+r}Q_{\mathcal{F},n+s}-Q_{\mathcal{L},n+s}Q_{\mathcal{F},n+r}=2(-q)^{n+r}\mathcal{F}_{s-r}(Q_{\mathcal{L},0}-[q]).$$
\end{proof}

After deriving these famous identities, we present some other identities for the generalized Fibonacci quaternions. In particular, when using the Binet formulas to obtain identities for the $(p,q)$-Fibonacci and $(p,q)$-Lucas quaternions, we require $\underline{\alpha}^{2}$ and $\underline{\beta}^{2}$. These products are given in the next lemma.
\begin{lemma}
We have
\begin{equation}\label{eq:8}
\underline{\alpha}^{2}=(Q_{\mathcal{L},0}-r_{p,q})+\Delta (Q_{\mathcal{F},0}-s_{p,q}),
\end{equation}
and
\begin{equation}\label{eq:9}
\underline{\beta}^{2}=(Q_{\mathcal{L},0}-r_{p,q})-\Delta (Q_{\mathcal{F},0}-s_{p,q}),
\end{equation}
where $\Delta=\alpha-\beta$, $r_{p,q}=1+\frac{p}{2}(\mathcal{F}_{2}+\mathcal{F}_{4}+\mathcal{F}_{6})+q(\mathcal{F}_{1}+\mathcal{F}_{3}+\mathcal{F}_{5})$, $s_{p,q}=\frac{1}{2}(\mathcal{F}_{2}+\mathcal{F}_{4}+\mathcal{F}_{6})$ and $\mathcal{F}_{n}$ is the $n$-th $(p,q)$-Fibonacci number.
\end{lemma}
\begin{proof}
From the definitions of $\underline{\alpha}$ and $\underline{\beta}$, and using $i^{2}=j^{2}=k^{2}=-1$, $ijk=-1$ and $\alpha^{n}=\mathcal{F}_{n}\alpha+q\mathcal{F}_{n-1}$ for $n\geq1$, we have
\begin{align*}
\underline{\alpha}^{2}&=2(1+\alpha i+\alpha^{2}j+\alpha^{3}k)-(1+\alpha^{2}+\alpha^{4}+\alpha^{6})\\
&=2+pi+(p^{2}+2q)j+(p^{3}+3pq)k+\Delta(i+pj+(p^{2}+q)k)\\
&\ \ -(1+(\mathcal{F}_{2}\alpha+q\mathcal{F}_{1})+(\mathcal{F}_{4}\alpha+q\mathcal{F}_{3})+(\mathcal{F}_{6}\alpha+q\mathcal{F}_{5}))\\
&=(Q_{\mathcal{L},0}-r_{p,q})+\Delta (Q_{\mathcal{F},0}-s_{p,q}),
\end{align*}
where $r_{p,q}=1+\frac{p}{2}(\mathcal{F}_{2}+\mathcal{F}_{4}+\mathcal{F}_{6})+q(\mathcal{F}_{1}+\mathcal{F}_{3}+\mathcal{F}_{5})$ and $s_{p,q}=\frac{1}{2}(\mathcal{F}_{2}+\mathcal{F}_{4}+\mathcal{F}_{6})$ and the final equation gives Eq. (\ref{eq:8}). The other can be computed similarly.
\end{proof}

We present some interesting identities for $(p,q)$-Fibonacci, $(p,q)$-Lucas quaternions and generalized Fibonacci quaternions.
\begin{theorem}
For any integer $n$, we have
\begin{equation}\label{equa:9}
Q_{\mathcal{L},n}^{2}-Q_{\mathcal{F},n}^{2}=\left( 
\begin{array}{c}
\frac{\Delta^{2}-1}{\Delta^{2}}(Q_{\mathcal{L},0}-r_{p,q})\mathcal{L}_{2n}+(Q_{\mathcal{F},0}-s_{p,q})\mathcal{F}_{2n}  \\ 
+2\frac{(\Delta^{2}+1)(-q)^{n}}{\Delta^{2}}(Q_{\mathcal{L},0}-[q]).
\end{array}%
\right)
\end{equation}
\end{theorem}
\begin{proof}
Using the Binet formulas for the $(p,q)$-Fibonacci and $(p,q)$-Lucas quaternions, we obtain
\begin{align*}
\Delta^{2}(Q_{\mathcal{L},n}^{2}-Q_{\mathcal{F},n}^{2})&=\Delta^{2}\left(\underline{\alpha}\alpha^{n}+\underline{\beta}\beta^{n}\right)^{2}-\left(\underline{\alpha}\alpha^{n}-\underline{\beta}\beta^{n}\right)^{2}\\
&=(\Delta^{2}-1)(\underline{\alpha}^{2}\alpha^{2n}+\underline{\beta}^{2}\beta^{2n})+(\Delta^{2}+1)(\alpha \beta)^{n}(\underline{\alpha}\underline{\beta}+\underline{\beta}\underline{\alpha}).
\end{align*}
Substituting Eqs. (\ref{eq:1}) and (\ref{eq:2}) into the last equation, we have
\begin{equation}\label{eq:10}
\Delta^{2}(Q_{\mathcal{L},n}^{2}-Q_{\mathcal{F},n}^{2})=(\Delta^{2}-1)(\underline{\alpha}^{2}\alpha^{2n}+\underline{\beta}^{2}\beta^{2n})+2(\Delta^{2}+1)(\alpha \beta)^{n}(Q_{\mathcal{L},0}-[q]).
\end{equation}
Then, using Eqs. (\ref{eq:8}) and (\ref{eq:9}), we obtain 
\begin{equation}\label{eq:11}
\underline{\alpha}^{2}\alpha^{2n}+\underline{\beta}^{2}\beta^{2n}=(\alpha^{2n}+\beta^{2n})(Q_{\mathcal{L},0}-r_{p,q})+\Delta (Q_{\mathcal{F},0}-s_{p,q})(\alpha^{2n}-\beta^{2n}).
\end{equation}
Substituting Eq. (\ref{eq:11}) into Eq. (\ref{eq:10}) gives Eq. (\ref{equa:9}).
\end{proof}

\begin{theorem}
For any integers $n$ and $m$, we have
\begin{equation}\label{equa:12}
Q_{\mathcal{F},n}Q_{w,m}-Q_{w,m}Q_{\mathcal{F},n}=2(-q)^{n+1}\omega W_{m-n},
\end{equation}
where $\omega=qi+pj-k$ and $W_{n}=\frac{A\alpha^{n}-B\beta^{n}}{\alpha-\beta}$ is the $n$-th generalized Fibonacci number.
\end{theorem}
\begin{proof}
The Binet formulas for the $(p,q)$-Fibonacci and generalized Fibonacci quaternions give
\begin{align*}
\Delta^{2}(Q_{\mathcal{F},n}Q_{w,m}&-Q_{w,m}Q_{\mathcal{F},n})\\
&=\left(\underline{\alpha}\alpha^{n}-\underline{\beta}\beta^{n}\right)\left(A\underline{\alpha}\alpha^{m}-B\underline{\beta}\beta^{m}\right)\\
&\ \ -\left(A\underline{\alpha}\alpha^{m}-B\underline{\beta}\beta^{m}\right)\left(\underline{\alpha}\alpha^{n}-\underline{\beta}\beta^{n}\right)\\
&=(A\alpha^{m}\beta^{n}-B\alpha^{n}\beta^{m})(\underline{\alpha}\underline{\beta}-\underline{\beta}\underline{\alpha}).
\end{align*}
Using Eqs. (\ref{eq:1}) and (\ref{eq:2}), we have
\begin{align*}
Q_{\mathcal{F},n}Q_{w,m}-Q_{w,m}Q_{\mathcal{F},n}&=(\alpha \beta)^{n}(A\alpha^{m-n}-B\beta^{m-n})(\underline{\alpha}\underline{\beta}-\underline{\beta}\underline{\alpha})\\
&=2(-q)^{n+1}\omega W_{m-n},
\end{align*}
where $\omega=qi+pj-k$ and $W_{n}$ is the $n$-th generalized Fibonacci number defined by $W_{n}=\frac{A\alpha^{n}-B\beta^{n}}{\alpha-\beta}$. 
\end{proof}

Taking $m=n$ in this theorem and using $W_{0}=a$, we obtain the next identity.
\begin{corollary}
For any integer $n$, we have
\begin{equation}\label{eq:13}
Q_{\mathcal{F},n}Q_{w,n}-Q_{w,n}Q_{\mathcal{F},n}=2(-q)^{n+1}a\omega,
\end{equation}
where $A=b-a\beta$, $B=b-a\alpha$ and $\omega=qi+pj-k$.
\end{corollary}


\end{document}